\documentclass{amsart}
\usepackage{amssymb}
\usepackage{enumerate}
\usepackage{url}
\usepackage{xypic}
\usepackage{graphicx}

\newcommand\N{\mathbb N}

\newcommand\R{\mathbb R}

\newcommand\A{\mathbb A}

\newcommand\ph\varphi
\newcommand\ps\psi
\newcommand\ep\varepsilon
\newcommand\rh\varrho
\newcommand\al\alpha
\newcommand\be\beta
\newcommand\ga\gamma
\newcommand\om\omega
\newcommand\ta\tau
\renewcommand\th\theta
\newcommand\de\delta
\newcommand\ze\zeta
\newcommand\ch\chi
\newcommand\et\eta
\newcommand\io\iota
\newcommand\la\lambda
\newcommand\si\sigma

\newcommand\dd\ddagger
\newcommand\Ga\Gamma
\newcommand\De\Delta
\newcommand\Th\Theta
\newcommand\La\Lambda
\newcommand\Si\Sigma
\newcommand\Ph\Phi
\newcommand\Ps\Psi
\newcommand\Om\Omega
\DeclareMathOperator{\PO}{PO}
\DeclareMathOperator{\QM}{QM}
\DeclareMathOperator{\Pos}{Pos}\DeclareMathOperator{\supp}{supp}
\DeclareMathOperator{\ord}{ord}
\newcommand\vv{\vee\vee}
\newcommand\V{\mathcal{V}}
\newcommand\Se{\mathcal{S}}

\newtheorem{theorem}{Theorem}[section]

\newtheorem{proposition}[theorem]{Proposition}
\newtheorem{corollary}[theorem]{Corollary}
\theoremstyle{definition}
\newtheorem{definition}[theorem]{Definition}

\newtheorem{example}[theorem]{Example}
\theoremstyle{remark}
\newtheorem{remark}[theorem]{Remark}

\newcommand\M{\mathfrak{m}}

\begin{document}
\title[Positive Polynomials and Sequential Closures of Quadratic Modules]
{Positive Polynomials and Sequential Closures of Quadratic
Modules}

\author{Tim Netzer}
\address{Universit\"at Konstanz, Fachbereich Mathematik und Statistik, 78457 Konstanz, Germany}
\email{tim.netzer@uni-konstanz.de}

\date{\today}

\begin{abstract} Let $\Se=\{x\in\R^n\mid f_1(x)\geq 0,\ldots,f_s(x)\geq 0\}$ be a basic closed semi-algebraic set in
$\R^n$ and $\PO(f_1,\ldots,f_s)$ the corresponding preordering in
$\R[X_1,\ldots,X_n]$. We examine for which polynomials $f$
 there exist identities $$f+\ep
q\in \PO(f_1,\ldots,f_s) \mbox{ for all } \ep>0.$$ These are
precisely the elements of the sequential closure of
$\PO(f_1,\ldots,f_s)$ with respect to the finest locally convex
topology. We solve the open problem from \cite{km,kms}, whether
this equals the double dual cone
$$\PO(f_1,\ldots,f_s)^{\vv},$$ by providing a counterexample. We
then prove a theorem that allows to obtain identities for
polynomials as above, by looking at a family of
\textit{fibre-preorderings}, constructed from bounded polynomials.
These fibre-preorderings are easier to deal with than the original
preordering in general. For a large class of examples we are thus
able to show that either \textit{every} polynomial $f$ that is
nonnegative on $\Se$ admits such representations, or at least the
polynomials from $\PO(f_1,\ldots,f_s)^{\vv}$ do. The results also
hold in the more general setup of arbitrary commutative algebras
and quadratic modules instead of preorderings.
\end{abstract}

\maketitle
\section{Introduction} Finitely many real polynomials
$f_1,\ldots,f_s\in\R[\underline{X}]=\R[X_1,\ldots,X_n]$ define a
basic closed semi-algebraic set
$$\Se=\Se(f_1,\ldots,f_s)=\left\{x\in\R^n\mid f_1(x)\geq
0,\ldots,f_s(x)\geq 0\right\}.$$ One is interested in finding
characterizations of $\Pos(\Se)$, the set of all polynomials that
are nonnegative on $\Se$. Obvious candidates for such nonnegative
polynomials are the elements of
$$\PO(f_1,\ldots,f_s):=\left\{\sum_{e\in\{0,1\}^s}\si_ef_1^{e_1}\cdots
f_s^{e_s}\mid \si_e\in\sum\R[\underline{X}]^2\right\},$$ the so
called \textit{preordering} generated by $f_1,\ldots,f_s$. Many
works in Real Algebra and Real Algebraic Geometry deal with the
question how $\PO(f_1,\ldots,f_s)$ relates to $\Pos(\Se)$. In
dimension one, equality occurs often, at least if the $f_i$ are
the so called \textit{natural generators} for $\Se$ (see
\cite{km,kms}). In dimension two, equality is a much more uncommon
phenomenon. For example, not every globally nonnegative polynomial
in two variables is a sum of squares of polynomials. However,
Scheiderer has given two local global principles that yield
equality between $\PO(f_1,\ldots,f_s)$ and $\Pos(\Se)$ under
certain conditions, in the case that $\Se$ is compact and
two-dimensional (see \cite{s1,s2} and also \cite{m,ckm}). The
results require a good behavior of the curves bounding $\Se$.
Noncompact two-dimensional examples where equality holds are even
more rare. One of the few known examples is the preordering
generated by $1-X^2$ in $\R[X,Y]$, see \cite{m2}. From dimension
three upwards, equality between $\PO(f_1,\ldots,f_s)$ and
$\Pos(\Se)$ is never possible. This is one of the main results
from \cite{s3}.

Of course one can ask questions beside equality. For example,
Schm\"udgen's famous result from \cite{sch1} says that in case
$\Se$ is compact, $\PO(f_1,\ldots,f_s)$ contains every  polynomial
which is \textit{strictly} positive on $\Se$, no matter what
generators $f_1,\ldots,f_s$ of $\Se$ are chosen, and independent
of the dimension of $\Se$. However, this result fails in general
if $\Se$ is not compact. If $\Se$ is very big,  for example if it
contains a full dimensional cone, then the preordering is indeed
far from containing every nonnegative or positive polynomial (see
for example \cite{km,kms,n2,ps,s}).

Another question arising in this context concerns the
$\Se$-\textit{Moment Problem}. One wants to find a
characterization of the linear functionals
$L\colon\R[\underline{X}]\rightarrow\R$ that are integration on
$\Se$. Haviland's Theorem (\cite{h}, see also \cite{m}) provides a
necessary and sufficient condition. Namely,  $L$ is integration on
$\Se$ if and only if $L$ is $\geq 0$ on $\Pos(\Se)$. As a
characterization of $\Pos(\Se)$ is not very simple either, one
wants to weaken the condition on $L$ and still apply Haviland's
Theorem. More precisely, one wants to know whether $L\geq 0$ on
$\PO(f_1,\ldots,f_s)$ is sufficient for $L$ to be integration. In
view of Haviland's Theorem, that means to ask whether
$$\PO(f_1,\ldots,f_s)^{\vv}=\Pos(\Se)$$ holds, where
$\PO(f_1,\ldots,f_s)^{\vv}$ denotes the \textit{double dual cone}
of the preordering. One says that the preordering has the
\textit{Strong Moment Property} (SMP) in this case. For example,
if $\Se$ is compact, then by the above mentioned Schm\"udgen
Theorem, the preordering always  has (SMP). There are also many
noncompact examples. \cite{sch2} gives a criterion to decide this
question, involving \textit{fibre-preorderings} constructed from
bounded polynomials.

Now in \cite{km,kms}, the following preordering is considered:
$$\PO(f_1,\ldots,f_s)^{\dd}:=\left\{f\in\R[\underline{X}]\mid
\exists q\in\R[\underline{X}]\ \forall \ep>0\ f+\ep
q\in\PO(f_1,\ldots,f_s)\right\}.$$ This definition does not use
linear functionals and is therefore better accessible to algebraic
methods. $\PO(f_1,\ldots,f_s)^{\dd}$ turns out to be the
\textit{sequential closure} of $\PO(f_1,\ldots,f_s)$ with respect
to the finest locally convex topology on $\R[\underline{X}]$,
whereas $\PO(f_1,\ldots,f_s)^{\vv}$ is the \textit{closure}. The
following relations hold:
$$\PO(f_1,\ldots,f_s)\subseteq \PO(f_1,\ldots,f_s)^{\dd}\subseteq
\PO(f_1,\ldots,f_s)^{\vv}\subseteq\Pos(\Se).$$ In \cite{sw},
Theorem 5.1,  it is shown that every element from $\Pos(\Se)$
belongs to the sequential closure of the preordering in a certain
\textit{localization} of the polynomial ring. A slightly weaker
version of that is \cite{m}, Theorem 6.2.3.

The authors of \cite{km,kms,p} prove
$\PO(f_1,\ldots,f_s)^{\dd}=\Pos(\Se)$ (and therefore (SMP)) for
certain classes of preorderings in the polynomial ring.
 The question
whether in polynomial rings
$$\PO(f_1,\ldots,f_s)^{\dd}=\PO(f_1,\ldots,f_s)^{\vv}$$ always
holds remained open in these works (see Open Problem 3 in
\cite{km} and Open Problem 2 in \cite{kms}). We solve this problem
by providing a counterexample; $\PO(f_1,\ldots,f_s)^{\dd}$ does
not equal $\PO(f_1,\ldots,f_s)^{\vv}$ in general.

Then it is natural to ask for conditions under which equality
holds. It is also interesting to ask whether a fibre theorem in
the spirit of \cite{sch2} can be established for
$\PO(f_1,\ldots,f_s)^{\dd}$ instead of
$\PO(f_1,\ldots,f_s)^{\vv}$. This question already turned up in
\cite{sch2}. Our counterexample implies that such a general
theorem will require stricter assumptions  than Schm\"udgen's
Fibre Theorem.

Theorem 5.3 from \cite{kms} is such a fibre theorem for finitely
generated preorderings that describe cylinders with compact cross
section. We will generalize this result to quadratic modules in
arbitrary commutative algebras.

We will then deduce criterions for
$$\PO(f_1,\ldots,f_s)^{\dd}=\PO(f_1,\ldots,f_s)^{\vv}$$ and
$$\PO(f_1,\ldots,f_s)^{\dd}=\Pos(\Se)$$ to hold.
They allow applications that go beyond the known examples from
\cite{km,kms,p}.

\bigskip\noindent

 {\bf \noindent Acknowledgements} The author thanks Murray
 Marshall and Claus Scheiderer for interesting and helpful
 discussions on the topic of this work.

\section{Notations and Preliminaries}
For the following results on topological vector spaces we refer to
\cite{b}, Chapter II. Let $E$ be an $\R$-vector space. The
\textit{finest locally convex topology} on $E$ is the vector space
topology having the collection of all convex, absorbing and
symmetric subsets of $E$ as a neighborhood base of zero. All the
linear functionals on $E$ are continuous, $E$ is Hausdorff, and
every finite dimensional subspace of $E$ inherits the canonical
topology. By \cite{sf}, Chapter 2, Exercise 7(b), a sequence in
$E$ converges if and only if it lies in a finite dimensional
subspace of $E$ and converges there. So for the \textit{sequential
closure} of a set $M$ in $E$, defined as the set of all limits of
sequences from $M$, and denoted by $M^{\dd}$, we have
$$M^{\dd}=\bigcup_{W} \overline{M\cap W},$$ where the union runs
over all finite dimensional subspaces $W$ of $E$.  From now on, we
will restrict ourself to \textit{convex cones} in $E$, i.e. to
subsets $M$ that are closed under addition and multiplication with
positive reals. The following alternative characterization for
$M^{\dd}$ can be found in \cite{cmn}: $$M^{\dd}=\left\{f\in E\mid
\exists q\in E\ \forall \ep>0\ f+\ep q\in M\right\}.$$ For
preorderings $M$ in $\R$-algebras, the object $M^{\dd}$ was first
introduced in \cite{km} in terms of this last characterization. It
only turned out later that it is the sequential closure of $M$.

The \textit{closure} of a convex cone $M$ is denoted by
$\overline{M}.$ \cite{b}, Chapter II.39, Corollary 5 implies
$\overline{M}=M^{\vv}$ for convex cones. Here, $M^{\vv}$ denotes
the double dual cone of $M$, i.e. the set
$$\left\{x\in E\mid L(x)\geq 0 \mbox{ for all } L\colon
E\rightarrow \R \mbox{ linear with } L(M)\subseteq\R_{\geq
0}\right\}.$$ We obviously have $M\subseteq M^{\dd}\subseteq
M^{\vv}.$ If $E$ has countable vector space dimension, then a set
$M$ in $E$ is closed if and only if its intersection with every
finite dimensional subspace of $E$ is closed, by \cite{bi},
Proposition 1. So $M$ is closed if and only if it is sequentially
closed, i.e. if $M=M^{\dd}$ holds. This implies that the
(transfinite) sequence of \textit{iterated sequential closures} of
$M$ terminates at $\overline{M}$ in that case. Note that in case
$E$ is not countable dimensional, then this sequence may terminate
before it reaches the closure. It indeed always terminates at the
closure with respect to the \textit{topology of finitely open
sets}, which can be smaller in the case of uncountable dimension.
We refer to \cite{cmn1} for more information.

Now let $A$ be a commutative $\R$-algebra with $1$. Of course $A$
is an $\R$-vector space, and we equip it with the finest locally
convex topology. A \textit{quadratic module} is a set $M\subseteq
A$ with $M+M\subseteq M, 1\in M$ and $A^2\cdot M\subseteq M.$
Here, $A^2$ denotes the set of squares in $A$. A
\textit{preordering} is a quadratic module with the additional
property $M\cdot M\subseteq M.$ Quadratic modules (and
preorderings) are convex cones, and we look at $M^{\dd}$ and
$M^{\vv},$ defined as above.

For a set $S\subseteq A$, the smallest quadratic
module/preordering containing $S$ is called the quadratic
module/preordering \textit{generated} by $S$. We write $\QM(S)$
and $\PO(S)$ for it, respectively. $\QM(S)$ consists of all finite
sums of elements $\si$ and $\si\cdot f,$ where $f\in S$ and $\si$
is a sum of squares in $A$. $\PO(S)$ consists of all finite sums
of elements $\si f_1\cdots f_t$, where $\si$ is a sum of squares,
$t\geq 0$ and all $f_i\in S$. Of special interest is the case that
$S$ is finite. We call $\QM(S)$ and $\PO(S)$ \textit{finitely
generated} then.

An important notion, introduced in \cite{ps}, is that of
\textit{stability} of a finitely generated quadratic module. If
$M$ is generated by $S=\{a_1,\ldots,a_s\}$, then we call $M$
\textit{stable}, if for every finite dimensional $\R$-subspace $U$
of $A$ there is another finite dimensional $\R$-subspace $V$ of
$A$, such that $$M\cap U\subseteq
\left\{\si_0+\si_1a_1+\cdots+\si_ta_s\mid \si_i\in\sum
V^2\right\}.$$ A map that assigns to each $U$ such a $V$ is then
called a \textit{stability map}. Whereas the notion of stability
is independent of the choice of generators of $M$, the stability
map is not. We refer the reader to \cite{n2,ps,s} for proofs and
details.

To $A$ there corresponds the set of all real characters on $A$,
i.e.
$$\V_A:=\left\{\al\colon A\rightarrow\R\mid \al \mbox{ unital } \R\mbox{-algebra
homomorphism}\right\}.$$ Elements $a$ from $A$ define functions
$\hat{a}$ on $\V_A$ by $\hat{a}(\al):=\al(a).$ We equip $\V_A$
with the coarsest topology making all these functions continuous.
As the functions $\hat{a}$  separate points, this makes $\V_A$ a
Hausdorff space, and we have the algebra homomorphism
$$\hat{} \colon A\rightarrow C(\V_A,\R).$$ If $A$ is finitely
generated as an $\R$-algebra, then $\V_A$ embeds into some $\R^n$,
 by taking generators $x_1,\ldots,x_n$ and sending $\al$ to
$(\al(x_1),\ldots,\al(x_n))$. So $\V_A$ is  the zero set in $\R^n$
of the kernel of the algebra homomorphism
$\pi\colon\R[X_1,\ldots,X_n]\rightarrow A; X_i\mapsto x_i$. The
use of an element $a$ from $A$ as a function $\hat{a}$ coincides
with the usual use of $a$ as a polynomial function on that
embedded variety. In particular, the topology on $\V_A$ is
inherited from the canonical one on $\R^n$. Note also that $A$ is
a countable dimensional $\R$-vector space in case it is finitely
generated as an $\R$-algebra.

Now we consider arbitrary commutative $\R$-algebras $A$ with $1$
again. For a set $M\subseteq A$, without loss of generality a
quadratic module, we define $$\Se(M)=\left\{\al\in \V_A\mid
\al(M)\subseteq\R_{\geq 0}\right\}.$$ If $M$ is finitely
generated, then $\Se(M)$ is called a \textit{basic closed
semi-algebraic set}. For any set $Y \subseteq\V_A$ we define
$$\Pos(Y):=\left\{a\in A\mid \hat{a}\geq 0 \mbox{ on }
Y\right\}.$$ This is a preordering. Starting with a quadratic
module or a preordering $M$ in $A$, we have the following chain:

$$M\subseteq M^{\dd}\subseteq M^{\vv}\subseteq\Pos(\Se(M)).$$ The
last inclusion comes from the fact that characters on $A$ are in
particular linear functionals. As for example proven in
\cite{cmn1}, $M^{\dd}$ and $M^{\vv}$ are again quadratic modules,
even preorderings if $M$ was a preordering. Following
\cite{km,kms,sch2}, we make the following definitions.

\begin{definition}\label{smpdd}\begin{itemize}\item[(i)] $M$ has the \textit{strong moment property}
(SMP), if $M^{\vv}=\Pos(\Se(M))$ holds \item[(ii)] $M$ has the
\textit{$\dd$-property}, if  $M^{\dd}=\Pos(\Se(M))$ holds
\item[(iii)] $M$ is \textit{saturated}, if $M=\Pos(\Se(M))$ holds
\item[(iv)] $M$ is \textit{closed}, if $M=M^{\vv}$ holds
\end{itemize}
\end{definition}

The interest in $M^{\vee\vee}$ and (SMP) is motivated by a
classical theorem by Haviland. For certain classes of algebras, it
states that a linear functional on $A$ is integration with respect
to some measure on $\Se(M)$, if and only if it is nonnegative on
$\Pos(\Se(M))$ (\cite{h} in the case that $A$ is a polynomial
algebra, see \cite{m} for a more general version). So if $M$ has
(SMP), then every functional that is nonnegative on $M$ is
integration on $\Se(M)$. Nonnegativity on $M$ is a priori a much
weaker condition than nonnegativity on $\Pos(\Se(M)).$ This is one
of the reasons that make (SMP) so interesting.

A method to decide whether (SMP) holds for a finitely generated
preordering $P$ in $\R[\underline{X}]$ is given in \cite{sch2}, as
mentioned in the introduction. Let $b$ be a polynomial that is
bounded on $\Se(P)$. Then $$P^{\vv}=\bigcap_{r\in\R}
\left(P+(b-r)\right)^{\vv}$$ holds, where $(b-r)$ denotes the
ideal generated by $b-r$. This implies that $P$ has (SMP) if and
only if all the preorderings $P+(b-r)$ have (SMP). As these so
called \textit{fibre preorderings} usually describe lower
dimensional semi-algebraic sets, they are easier to deal with in
general.
 The result in particular implies that every finitely generated
preordering in $\R[\underline{X}]$ describing a compact set
$\Se(P)$ has (SMP). This was already part of the earlier paper
\cite{sch1}. For an alternative proof of the result from
\cite{sch2} see also \cite{m,n}.

The $\dd$-property was introduced and studied in \cite{km,kms} for
polynomial algebras. The authors for example show that in case
$\Se(P)$ is a cylinder with compact cross section, then the
preordering $P$ has the $\dd$-property, under reasonable
assumptions on the generators of $P$. This is also shown, using
different methods, in \cite{p}.

It was an open problem in \cite{km,kms}, whether (SMP) and the
$\dd$-property are equivalent or even $M^{\dd}=M^{\vv}$ is always
true. We start by showing that the answer to both questions is
negative.

\section{A counterexample}
The example in this section will answer Open Problem 3 in
\cite{km} and Open Problem 2 in \cite{kms} to the negative. It
will also give a negative answer to the question in \cite{sch2},
whether the fibre theorem holds for the $\dd$-property instead of
(SMP).

Consider $A=\R[X,Y]$, the polynomial algebra in two variables, and
$f_1=Y^3,\ f_2=Y+X,\ f_3=1-XY\ \mbox{ and } f_4=1-X^2\in A.$ These
polynomials define the following basic closed semi-algebraic set
$\Se(f_1,\ldots,f_4)$ in $\V_A=\R^2$:

\begin{center}\bf\includegraphics[scale=0.3]{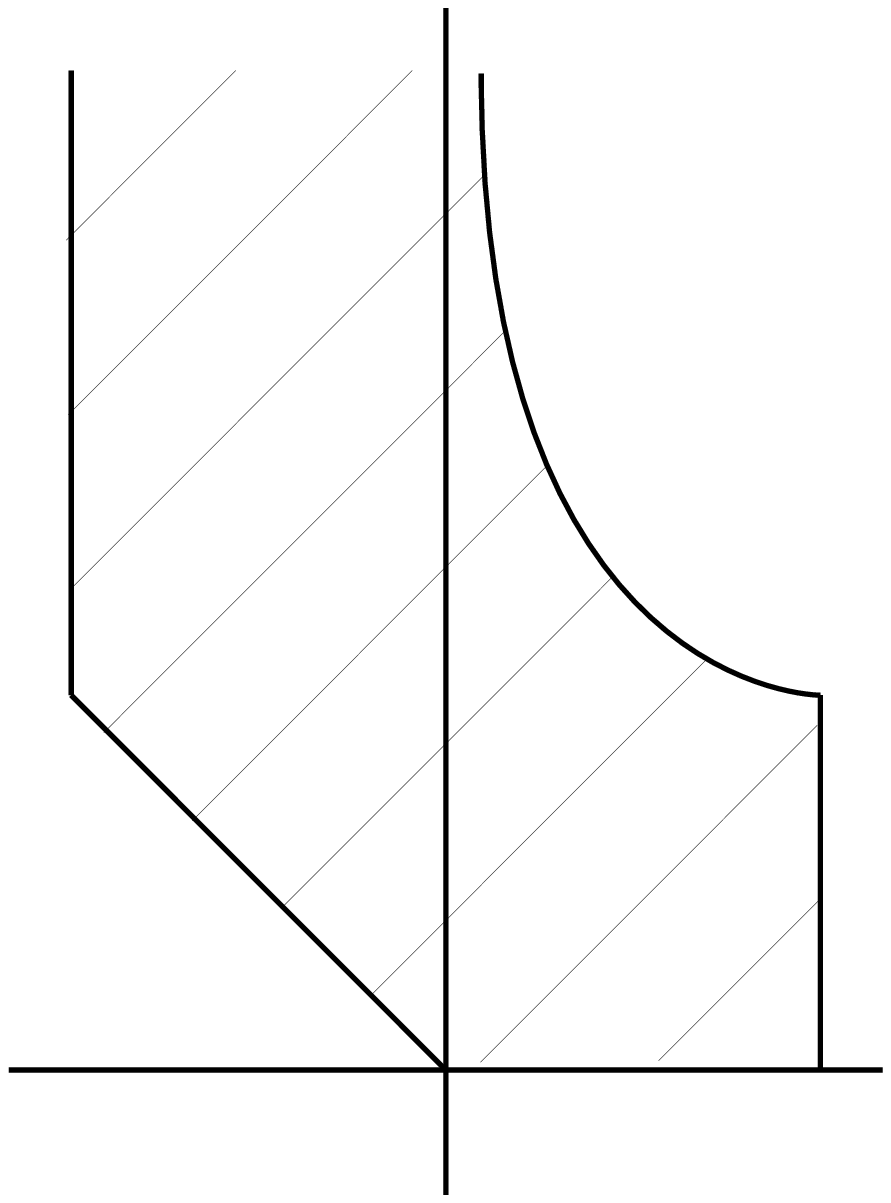}\end{center}

\begin{proposition}\label{smp} The preordering $P=\PO(f_1,f_2,f_3,f_4)$ in
$\R[X,Y]$ has (SMP)
\end{proposition}
\begin{proof} The polynomial $X$ is bounded on $\Se(f_1,\ldots,f_4)$,  and we can
apply Schm\"udgen's fibre theorem from \cite{sch2} to $P$.
 For any $r\in \R\setminus[-1,0]$, the preordering
$$P_{r}:= P+ (X-r)=\PO(f_1,\ldots,f_4,X-r,r-X)$$ describes a compact semi-algebraic set
and therefore has (SMP) by \cite{sch1} (even the
($\ddagger$)-property). For $r\in [-1,0]$, the preordering
$\PO(Y^3,Y+r,1-r Y)\subseteq \R[Y]$ is saturated. This follows
from \cite{km}, Theorem 2.2. So  $P_r=\PO(f_1,\ldots,f_4,
X-r,r-X)$ in $\R[X,Y]$ is saturated as well. In particular, $P_r$
has (SMP). So by \cite{sch2}, the whole preordering $P$ has (SMP).
\end{proof}
The next result is a characterization of $P^{\ddagger}$. We write
$$\PO(a_1,\ldots,a_s)_d$$ for the set of elements
having a representation in $\PO(a_1,\ldots,a_s)$ with sums of
squares of elements of degree $\leq d$.

\begin{proposition}\label{dd}A polynomial $f\in\R[X,Y]$ belongs to
$\PO(f_1,f_2,f_3,f_4)^{\ddagger}$ if and only if there is some
$d\in\N$ such that for all $r\in[-1,1]$, $f(r,Y)$ belongs to
$\PO(f_1(r,Y),\ldots,f_4(r,Y))_d$ in $\R[Y]$.
\end{proposition}
\begin{proof}
The "if"-part is a consequence of Theorem $\ref{ddapp}$ below (or
can already be obtained by looking at the proof of Theorem 5.3. in
\cite{kms}).

For the "only if"-part assume $f$ belongs to
$\PO(f_1,\ldots,f_4)^{\ddagger}$. So there is some $q\in\R[X,Y]$
and sums of squares $\si_e^{(\ep)}\in\sum\R[X,Y]$ for all $\ep>0$
and $e\in\{0,1\}^4$ such that $$f+\ep q = \sum_e
\si_e^{(\ep)}f_1^{e_1}\cdots f_4^{e_4}.$$ Note that the total
degree of the $\si_e^{(\ep)}$ may rise with $\ep$ getting smaller.
However, the degree as polynomials in $Y$ cannot rise; it is
bounded by the $Y$-degree of $f+\ep q$, which does not change with
$\ep$. This is because the set $\Se(f_1,\ldots,f_4)$ contains the
cylinder $[-1,0]\times [1,\infty]$, so whenever one adds two
polynomials which are nonnegative on it, the leading terms as
polynomials in $Y$ cannot cancel. So the degree in $Y$ of the sum
is the maximum of the $Y$-degrees of the terms.

By evaluating in $X=r$, this means that $f(r,Y)+\ep q(r,Y)$
belongs to $$\PO(f_1(r,Y),\ldots,f_4(r,Y))_d$$ for some fixed $d$
and all $r\in[-1,1], \ep>0$. But by \cite{ps}, Proposition 2.6,
this is a closed set in a finite dimensional subspace of $\R[Y]$.
So we get $f(r,Y)\in \PO(f_1(r,Y),\ldots,f_4(r,Y))_d$ for all
$r\in[-1,1]$, the desired result.
\end{proof}

\begin{corollary}
The preordering $P=\PO(f_1,\ldots,f_4)$ does not have the
$\ddagger$-property.
\end{corollary}
\begin{proof}
The polynomial $Y$ is obviously nonnegative on
$\Se(f_1,\ldots,f_4)$. However, it does not belong to the
preordering
$$\PO(f_1(1,Y),\ldots,f_4(1,Y))=\PO(Y^3,Y+1,1-Y)\subseteq \R[Y].$$
Indeed, writing down a representation and evaluating in $Y=0$,
this shows that $Y^2$ divides $Y$, a contradiction. So in view of
Proposition $\ref{dd}$, $Y$ can not belong to $P^{\ddagger}$.
\end{proof}
\begin{remark}
Note that $Y$ is not in $P^{\ddagger}$ as it fails to be in the
preordering corresponding to the fibre $X=1$. However, Proposition
$\ref{dd}$ even demands all the polynomials $f(r,Y)$ to have
representations in the fibre-preorderings
$$\PO(f_1(r,Y),\ldots,f_4(r,Y))$$ with \textit{simultaneous}
degree bounds, for $f$ to be in $P^{\ddagger}$. Indeed, there are
examples of polynomials belonging to all of the
fibre-preorderings, but failing the degree-bound condition (and so
also not belonging to $P^{\ddagger}$). We will give one here, as
it gives a justification for one of the assumptions in Theorems
\ref{main} and \ref{ddapp} below.
\end{remark}

\begin{example}\label{exam}
Take $f= 2Y+X$, which belongs to $\Pos(\Se(f_1,\ldots,f_4))$. For
any $r\in[-1,1]$, $f(r,Y)=2Y+r$ belongs to
$\PO(f_1(r,Y),\ldots,f_4(r,Y))$; for $r>0$ as $f(r,Y)$ is strictly
positive on the corresponding compact semi-algebraic set (so use
\cite{sch1}), for $r\in[-1,0]$, the fibre preordering is
saturated, as mentioned in the proof of Proposition $\ref{smp}$.

However, for $r\searrow 0$, there can be no bound on the degree of
the sums of squares in the representation. Indeed, for $r>0$,
write down a representation \begin{equation}\label{re}2Y+r
=\sum_{e\in\{0,1\}^3} \si_{e}^{(r)}Y^{3e_1}(Y+r)^{e_2}(1-r
Y)^{e_3},\end{equation} where the $\si_e^{(r)}$ are sums of
squares. Evaluating in $Y=0$, this shows
\begin{equation}\label{ab}\si_{(0,1,0)}^{(r)}(0)+\si_{(0,1,1)}^{(r)}(0)\leq
1.\end{equation} Now if the degrees of the $\si_e^{(r)}$ could be
bounded for all $r>0$, we could write down a first order logic
formula saying that we have representations as in ($\ref{re}$) for
all $r>0$. We add the statement ($\ref{ab}$) to the formula. By
Tarski's Transfer Principle, it holds in any real closed extension
field of $\R$. So take such a representation in some
non-archimedean real closed extension field $R$ for some $r>0$
which is infinitesimal with respect to $\R$. The same argument as
for example in \cite{kms}, Example 4.4. (a) shows that we can
apply the residue map $\mathcal{O}\rightarrow
\mathcal{O}/\mathfrak{m}=\R$ to the coefficients of all the
polynomials occurring in this representation. Here, $\mathcal{O}$
denotes the convex hull of $\R$ in $R$. This is a valuation ring
with maximal ideal $\mathfrak{m}$.

So we get a representation
\begin{align*}2Y =& \ \si_{(0,0,0)}+\si_{(1,0,0)}Y^3+\si_{(0,1,0)}Y +
\si_{(0,0,1)} +\si_{(1,1,0)}Y^4\\&+\si_{(1,0,1)}Y^3
+\si_{(0,1,1)}Y +\si_{(1,1,1)}Y^4\end{align*} with sums of squares
$\si_e$ in $\R[Y]$ fulfilling
\begin{equation}\label{con}\si_{(0,1,0)}(0)+\si_{(0,1,1)}(0)\leq 1.\end{equation} As no cancellation
of highest degree terms can occur, we get
$$0=\si_{(0,0,0)}=\si_{(1,0,0)}=\si_{(0,0,1)}=\si_{(1,1,0)}=\si_{(1,0,1)} = \si_{(1,1,1)}$$ as
well as $$\si_{(0,1,0)}+\si_{(0,1,1)}= 2.$$ This last fact
obviously contradicts ($\ref{con}$).

So for $2Y+X$, the degree bound condition on the fibres fails,
although the polynomial belongs to all of the fibre preorderings.
In view of Proposition $\ref{dd}$, it does not belong to
$\PO(f_1,\ldots,f_4)^{\ddagger}$. This shows that the "degree
bound"-assumption in Theorems $\ref{main}$ and \ref{ddapp} below
is really necessary.

Note also that the example is an explicit illustration of Theorem
8.2.6 from \cite{pd}, where the general impossibility of a certain
degree bound for Schm\"udgen's Theorem from \cite{sch1} is proved.
\end{example}

\begin{remark}
The above example answers the question in \cite{sch2}, whether the
fibre theorem holds for the $\ddagger$-property instead of (SMP).
Indeed, we have shown in the proof of Proposition $\ref{smp}$ that
all the fibre preorderings $P_{r}$ do not only have (SMP), but
even the $\ddagger$-property. As $P$ itself does not have the
$\ddagger$-property, this gives a negative answer to the question.
However, we will prove a result below that allows to use a
dimension reduction  when examining the $\ddagger$-property  under
certain conditions.

\end{remark}

\section{A Fibre Theorem for Sequential Closures}\label{haupt}

We want to prove a fibre theorem in the spirit of \cite{sch2}, to
be able to examine the sequential closure of a quadratic module in
terms of (easier) fibre-modules. We consider the following general
setup. Let $A,B$ be commutative $\R$-algebras with $1$, let $X$ be
a compact Hausdorff space, and assume we have algebra
homomorphisms $\ph\colon B\rightarrow A$ and $\ \hat{} \colon
B\rightarrow C(X,\R)$ (homomorphisms are always assumed to map $1$
to $1$):

$$\xymatrix{ A  \\  B \ar@{->}^{\widehat{}\quad}[r] \ar@{->}^{\ph}[u] & C(X,\R)}$$

We assume that the image of $B$ in $C(X,\R)$ separates points of
$X$, i.e. for any two distinct points $x,y\in X$ there is some
$b\in B$ such that $\hat{b}(x)\neq\hat{b}(y).$ The
Stone-Weierstrass Theorem  implies that $\widehat{B}$ is dense in
$C(X,\R)$ with respect to the sup-norm then.

 For $x\in X$ let $I_x:=\{b\in
B\mid\hat{b}(x)=0\}$ be the vanishing ideal of $x$ in $B$, and
$J_x$ the ideal in $A$ generated by $\ph(I_x)$. The homomorphism
$\ph$ makes $A$ a $B$-module in the usual sense of commutative
algebra (not to be confused with quadratic modules!). For a
$B$-submodule $W$ of $A$
 we write $$J_x(W)=\left\{\sum_{i=1}^n w_i
\ph(b_i)\mid n\in\N, w_i\in W, b_i\in I_x\right\}.$$ So $J_x(W)$
is a $B$-submodule of $W$. We have $J_x(A)=J_x$.

Now let $M\subseteq A$ be a quadratic module. For any $x\in X$,
$M+ J_x$ is again a quadratic module, called the
\textit{fibre-module} of $M$ with respect to $x$. Our first goal
is to prove the following abstract fibre theorem:

\begin{theorem}\label{main} Let $A,B$ be commutative
$\R$-algebras, $X$ a compact Hausdorff space, $\ph\colon
B\rightarrow A$ and $\ \hat{} \colon B\rightarrow C(X,\R)$ algebra
homomorphisms, such that $\hat{B}$ separates points of $X$. Let
$M\subseteq A$ be a quadratic module and assume $\ph(b)\in M$
whenever $\hat{b}>0$ on $X$. Then for any finitely generated
$B$-submodule $W$ of $A$ we have $$\bigcap_{x\in X} M+
J_x(W)\subseteq M^{\dd}.$$
\end{theorem}
 The requirement on $W$ can be understood as a "degree
 bound condition" as in Proposition \ref{dd}. Example \ref{exam} shows that $\bigcap_{x\in X}
 M+J_x\subseteq M^{\dd}$ is not true under the remaining
 assumptions in general (we will discuss this in more detail
 below).

 To prove the Theorem, we first need the following technical
Proposition.

\begin{proposition}\label{loc} Let $A,B$ be
$\R$-algebras and $\ph\colon B\rightarrow A$ an algebra
homomorphism. Let $X$ be a compact Hausdorff space  and $\
\hat{}\colon B\rightarrow C(X,\R)$ an algebra homomorphism whose
image separates points of $X$. Assume $a,a_1,\ldots,a_l\in A$ and
$\ep>0$ are such that for all $x\in X$ there is an identity
$$a=\sum_{i=1}^l \ph(b_i^{(x)})\cdot a_i,$$ with $b_i^{(x)}\in B$
and $|\widehat{b_i^{(x)}}(x)| < \ep$ for all $i$. Then there are
$b_1,\ldots,b_l\in B$ with $|\hat{b}_i| < \ep$ on $X$ for all $i$
and $$a=\sum_{i=1}^{l} \ph(b_i)\cdot a_i.$$
\end{proposition}
\begin{proof}

 Every $x\in X$ has an
open neighborhood $U_x$, such that $|\widehat{b_i^{(x)}}|<\ep$ on
$U_x$ for all $i=1,\ldots,l.$ By compactness of $X$ there are
$x_1,\ldots,x_t\in X$, such that $$X= U_{x_1}\cup\cdots\cup
U_{x_t}.$$ If $t=1$, then the result follows, so assume $t\geq 2.$
Choose a partition of unity $e_1,\ldots,e_t$ subordinate to that
cover, i.e. all $e_k$ are continuous functions from $X$ to
$[0,1]$, $\supp(e_k)\subseteq U_{x_k}$ for all $k$, and
$e_1(x)+\cdots + e_t(x)=1$ for all $x\in X.$ Then for
$$f_i:= e_1\cdot \widehat{b_i^{(x_1)}} + \cdots +
e_t\cdot\widehat{b_i^{(x_t)}}$$ we have $$\parallel f_i
\parallel<\ep,$$ where $\parallel\ \parallel$ denotes the sup-norm on
$C(X,\R)$. Let $$\de:=\min\left\{\ \ep- \parallel f_i
\parallel\ \mid i=1,\ldots,l\right\}$$ and choose a positive real number $N$, big enough to
bound the sup-norm of all $\widehat{b_i^{(x_k)}}$.

The image of $B$ in $C(X,\R)$ is dense, by the Stone-Weierstrass
Theorem. So we find $q_1,\ldots,q_{t-1}\in B$ such that $$
\parallel e_k - \hat{q}_k\parallel < \frac{\de}{N(t-1)t}$$
for $k=1,\ldots,t-1$, and we define
$$q_t:=1-\sum_{k=1}^{t-1}q_k.$$ So we have for $k=1,\ldots, t$ $$\parallel e_k - \hat{q}_k \parallel < \frac{\de}{Nt}.$$
We define $$b_i:= q_1\cdot b_i^{(x_1)} + \cdots + q_t\cdot
b_i^{(x_t)}$$ for $i=1,\ldots,l.$ So
\begin{align*}\parallel \hat{b}_i
\parallel & \leq\  \parallel f_i\parallel+ \parallel \hat{b}_i -f_i
\parallel \\ & \leq\ \parallel f_i\parallel+ \sum_{k=1}^t \parallel e_k-\hat{q}_k
\parallel\cdot\parallel \widehat{b_i^{(x_k)}}\parallel \\
& <\ \parallel f_i\parallel +\ \de\\ & \leq\ \ep.\end{align*} Now
as $\sum_{k=1}^t q_k =1$ we have

\begin{align*} a&=  \ph(\sum_{k=1}^t q_k)\cdot a \\& =
\sum_{k=1}^t\left( \ph(q_k)\cdot\sum_{i=1}^l
\ph(b_i^{(x_k)})a_i\right)\\& =\sum_{i=1}^l \ph\left(\sum_{k=1}^t
q_kb_i^{(x_k)}\right)\cdot a_i \\ & =\sum_{i=1}^l\ph(b_i)\cdot
a_i,
\end{align*} which proves the proposition.
\end{proof}

Now we can give the proof of Theorem \ref{main}. It is a
generalization of the proof of Theorem 5.3 from \cite{kms}.

\begin{proof}[Proof of Theorem \ref{main}] Fix a finitely generated $B$-submodule $W$ of $A$. Assume $f\in A$ has a representation $$f=m_x+j_x$$
with $m_x\in M$ and $j_x\in J_x(W),$ for all $x\in X.$ As
$J_x(W)\subseteq W$, we can assume without loss of generality
$m_x\in M\cap W$ for all $x$. Let $a_1,\ldots,a_l$ be generators
of $W$ as a $B$-module. Due to the identity $a=(\frac{a+1}{2})^2 -
(\frac{a-1}{2})^2$ we can assume that all $a_j$ are squares in $A$
(by possibly enlarging $W$). We will now show
$$f+\ep\sum_{i=1}^la_i\in M$$ for all $\ep>0.$ Therefore fix one
such $\ep>0.$ We take representations
$$j_x=\sum_{i=1}^l \ph\left(c_i^{(x)}\right)\cdot a_i,\quad m_x=\sum_{i=1}^l \ph\left(d_i^{(x)}\right)\cdot a_i $$
where all $c_i^{(x)},d_i^{(x)}\in B$ and
$\widehat{c_i^{(x)}}(x)=0$. Now each $x\in X$ has an open
neighborhood $U_x$, such that
$$| \widehat{c_i^{(x)}}|<\frac{\ep}{2} \mbox{ on } U_x$$ for
$i=1,\ldots,l.$ By compactness of $X$ we have
$$X=U_{x_1}\cup\cdots\cup U_{x_t}$$ for some $x_1,\ldots,x_t\in
X.$ Let $e_1,\ldots,e_t$ be a continuous partition of unity
subordinate to that cover. Using the Stone-Weierstrass Theorem, we
approximate the square root of each $e_k$ (which is again a
continuous function) by elements $g_k$ from $B$, such that
$$\sum_{k=1}^t \parallel e_k -\hat{g}_k^2\parallel \cdot
\parallel \widehat{d_i^{(x_k)}}\parallel\ <\frac{\ep}{2}$$ holds for all
$i=1,\ldots,l.$ Here, $\parallel \ \parallel$ denotes the sup-norm
on $C(X,\R)$ again. Define
$$a=f-\underbrace{\sum_{k=1}^t \ph(g_k)^2 \cdot m_{x_k}}_{\in
M}.$$ The proof is complete if we show $a+\ep\sum_{i=1}^la_i\in
M.$ Fix $x\in X$. Then \begin{align*} a &=\sum_{k=1}^t e_k(x)\cdot f - \sum_{k=1}^t \ph(g_k^2)\cdot m_{x_k}   \\
&= \sum_{k=1}^t e_k(x) \cdot (\underbrace{f -m_{x_k}}_{=j_{x_k}})
+\sum_{k=1}^t \left( e_k(x)-\ph(g_k^2)\right) m_{x_k} \\ &=
\sum_{k=1}^t e_k(x)\sum_{i=1}^l \ph\left( c_i^{(x_k)}\right) a_i +
\sum_{k=1}^t \left(e_k(x)-\ph(g_k^2)\right) \sum_{i=1}^l
\ph\left(d_i^{(x_k)}\right)a_i \\ &= \sum_{i=1}^l \left(
\sum_{k=1}^t e_k(x)\ph(c_i^{(x_k)}) \right)\cdot a_i  \\ & \qquad
+ \sum_{i=1}^l \left( \sum_{k=1}^t \left(e_k(x)
-\ph(g_k^2)\right)\ph(d_i^{(x_k)}) \right)\cdot a_i \\
&=\sum_{i=1}^l \ph(b_i^{(x)})\cdot a_i,
\end{align*}
where we define $$b_i^{(x)}=\sum_{k=1}^t e_k(x)\cdot c_i^{(x_k)} +
\left(e_k(x)-g_k^2\right)\cdot d_i^{(x_k)}.$$ By the above
considerations we have $$|\widehat{b_i^{(x)}}(x)|<\ep$$ for all
$i$. So we can apply Proposition \ref{loc} to $a,a_1,\ldots,a_l$
and find
$$a=\sum_{i=1}^l\ph(b_i)\cdot a_i$$ for some $b_i\in B$ with $|\hat{b}_i|<\ep$ on $X$. Thus $$a+\ep\sum_{i=1}^l a_i = \sum_{i=1}^l\ph(b_i+\ep)\cdot a_i \in M,$$
as all $\widehat{b_i+\ep}$  are strictly positive on $X$ and all
$a_i$ are squares.\end{proof}

We demonstrate how to apply Theorem \ref{main}, for a given
algebra $A$ and a quadratic module $M\subseteq A$. Therefore
assume there are $b_1,\ldots,b_t\in A$ with
$\La_i-b_i,b_i-\la_i\in M$ for real numbers $\la_i\leq \La_i$
($i=1,\ldots,t)$. This of course implies that each $b_i$ is
bounded as a function on $\Se(M)\subseteq \V_A$, but the converse
is not always true. Let $B=\R[b_1,\ldots,b_t]$ be the subalgebra
of $A$ generated by the $b_i$ and let $\ph\colon B\rightarrow A$
be the canonical inclusion. Let $\widetilde{M}$ be the quadratic
module in $B$ generated by $\La_1-b_1,b_1-\la_1,\ldots, \La_t-b_t,
b_t-\la_t.$ It is \textit{archimedean}, for example by \cite{pj},
Theorem 4.1. Let $X=\Se(\widetilde{M})\subseteq\V_B$, so $X$ is
compact, and the canonical homomorphism $\hat{}\colon B\rightarrow
C(X,\R)$ fulfills the separating points condition. Now whenever
some $\hat{b}$ is strictly positive on $X$, then $b\in
\widetilde{M},$ by \cite{j}, Theorem 6 (see also \cite{pd} Theorem
5.3.6 and \cite{m}, Theorem 5.4.4). So $\ph(b)\in M$. For any
$x\in X$,
 we have
$\hat{b}_i(x)\in[\la_i,\La_i]$, and $$J_x(W)=\left\{\sum_{i=1}^t
(b_i-\hat{b}_i(x))w_i\mid w_i\in W\right\}$$ holds for any
$B$-module $W$. Thus write for
$r=(r_1,\ldots,r_t)\in\La=\prod_{i=1}^t[\la_i,\La_i]$
$$J_{r}(W)=\left\{\sum_{i=1}^t (b_i-r_i)w_i\mid w_i\in
W\right\}$$ and $J_{r}:=J_{r}(A)=(b_1-r_1,\ldots,b_t-r_t)$. If $M$
is finitely generated as a quadratic module, then $M+J_r$ is also
finitely generated, by the generators of $M$ and $\pm
(b_1-r_1),\ldots, \pm(b_t-r_t)$. The following fibre theorem is
our main result.

\begin{theorem}\label{ddapp} Let $A$ be a commutative $\R$-algebra and $M\subseteq A$ a
quadratic module. Suppose $b_1,\ldots,b_t\in A$ are such that
$$\La_1-b_1,b_1-\la_1,\ldots,\La_t-b_t,b_t-\la_t\in M$$ for some real numbers $\la_i\leq \La_i$ ($i=1,\ldots,t)$.  Then for every finitely
generated $\R[b_1,\ldots,b_t]$-submodule $W$ of $A$ we have
$$\bigcap_{r\in\Lambda} M + J_{r}(W)
\subseteq M^{\dd},$$ where $\Lambda=\prod_{i=1}^t [\la_i,\La_i].$
In particular, if $M$ is finitely generated and all the (finitely
generated) quadratic modules $M+J_{r}$ are closed and stable with
the same stability map, then $M^{\dd}=\overline{M}.$ If all
$M+J_{r}$ are saturated and stable with the same stability map,
then $M$ has the $\dd$-property. (Here, the stability map with
respect to the canonical generators of each $M+J_{r}$ is meant.)
\end{theorem}

\begin{proof} The first part of the theorem is clear from the
above considerations and Theorem \ref{main}.  For the second part,
assume $M$ is finitely generated, say by $f_1,\ldots,f_s$. Then
$M+J_{r}$ is finitely generated as a quadratic module, by the
canonical generators
$$f_1,\ldots,f_s,\pm(b_1-r_1),\ldots,\pm(b_t -r_t).$$
Assume all $M+J_{r}$ are closed (or saturated, respectively) and
stable with the same stability map. Suppose some $f$ belongs to
$\overline{M}$ (or $\Pos(\Se(M))$, respectively). Then $f$ belongs
to all $\overline{M+J_{r}}$ (or $\Pos(\Se(M+J_{r}))$,
respectively), so to all $M+J_{r}$ by our assumption. Now by the
assumed stability there is a \textit{fixed} finite dimensional
$\R$-subspace $W$ of $A$, such that $f$ belongs to all $M +
J_{r}(W)$. So the first part of the theorem yields $f\in M^{\dd}.$
\end{proof}

\begin{remark}\label{rem}(1) In Example \ref{exam}, the polynomial $f=2Y+X$
belongs to all the preorderings $\PO(f_1,\ldots,f_4)+(X-r)$ in
$A=\R[X,Y]$. However, there is no finitely generated
$\R[X]$-submodule $W$ of $\R[X,Y]$ such that $f$ belongs to all
$\PO(f_1,\ldots,f_4)+J_r(W)$. This follows from what we have shown
in Example \ref{exam}. We have also seen  that $f$ does not belong
to $\PO(f_1,\ldots,f_4)^{\dd}.$  Note that $X+1,1-X$ in
$\PO(f_1,\ldots,f_4)$ is fulfilled, as $f_4=1-X^2$, and using an
easy calculation as for example in \cite{km}, Note 2.3 (4). So the
"degree bound condition" is necessary in Theorem \ref{ddapp} and
also in Theorem \ref{main}.

(2) Example \ref{einsdurch} below will show that the assumption
$\La_i-b_i,b_i-\la_i\in M$ for all $i$ can not be omitted in
Theorem \ref{ddapp}. So the same is true for the assumption
$$\hat{b}>0 \mbox{  on } X \Rightarrow \ph(b)\in M$$ in Theorem
\ref{main}.

(3) In case that all the occurring quadratic  fibre-modules
$M+J_r$ in Theorem \ref{ddapp} are saturated and stable with the
same stability map, we get a little bit more than the
$\dd$-property for $M$. We obtain that for every finite
dimensional subspace $V$ of $A$ there is some $q_V\in A$ such that
whenever $f\in\Pos(\Se(M))\cap V$, then $f+\ep q_V\in M$ for all
$\ep>0.$ In other words, the polynomial $q$ from the
$\dd$-property does only depend on the subspace $f$ is taken from,
not on the explicit choice of $f$. This follows from the proof of
Theorem \ref{main}, noting that $q$ does only depend on the
$B$-module $W$, which depends only on $V$ and the stability map
here.
\end{remark}

\section{Applications and Examples}

In this section we give some applications of the fibre theorem.
The first one is the Cylinder Theorem (Theorem 5.3 combined with
Corollary 5.5) from \cite{kms}. See \cite{km,kms} for the
definition of natural generators for semi-algebraic subsets of
$\R$.

\begin{corollary}\label{cylinder} Let $P=\PO(f_1,\ldots,f_s)$ be a finitely generated
preordering in the polynomial ring $\R[X_1,\ldots,X_n,Y]$. Assume
$N-\sum_{i=1}^{n}X_i^2\in P$ for some $N > 0$. Now for all $r\in
\R^n$, the preordering
$$\PO(f_1(r,Y),\ldots,f_s(r,Y))\subseteq \R[Y]$$ describes a basic closed semi-algebraic set $S_{r}$ in $\R$. Suppose the natural
generators for $S_{r}$ are among the $f_1(r,Y),\ldots, f_s(r,Y)$,
whenever $S_{r}$ is not empty. Then $P$ has the $\dd$-property.

If all the fibre sets $S_{r}$ are of the form
$\emptyset,(-\infty,\infty), (-\infty,p],$ $[q,\infty),
(-\infty,p]\cup[q,\infty)$ or $[p,q]$, then the result holds with
$P$ replaced by $M=\QM(f_1,\ldots,f_s).$

\end{corollary}
\begin{proof} The assumptions imply that all the preorderings
$$P+(X_1-r_1,\ldots,X_n-r_n)$$ (or the corresponding quadratic modules, respectively) are saturated and stable with
the same stability map for all $r$. See \cite{kms}, Section 4. An
easy calculation, as for example in \cite{km}, Note 2.3 (4), shows
$$\sqrt{N}-X_i,X_i+\sqrt{N}\in P$$ for all $i$. So we can apply
Theorem \ref{ddapp}.
\end{proof}

We can also use Theorem \ref{ddapp} in the case that the natural
generators are not among the $f_i(r,Y)$. This can be seen as a
generalization of Corollary 5.4 from \cite{kms}:

\begin{corollary} Let $M=\QM(f_1,\ldots,f_s)$ be a finitely
generated quadratic module in $\R[X_1,\ldots,X_n,Y]$ and assume
$N-\sum_{i=1}^n X_i^2\in M$ for some $N > 0.$ Suppose for all
$r\in\R^n$ the set $S_{r}$ (defined as in Corollary
\ref{cylinder}) is either empty or unbounded. Then
$$M^{\dd}=\overline{M}$$ holds.
\end{corollary}
\begin{proof}Again $\sqrt{N}-X_i,X_i+\sqrt{N}\in P$ for
all $i$. Furthermore, the assumptions imply that all the quadratic
modules
$$M+ (X_1-r_1,\ldots,X_n-r_n)$$ are closed and stable with the
same stability map for all $r$ (for the empty fibers use Theorem
4.5 from \cite{kms}). Now apply Theorem \ref{ddapp}.
\end{proof}

We want to get results for more complicated fibres. \cite{s1}
gives a criterion for quadratic modules on curves to be stable and
closed. However, we need some result to obtain the
\textit{uniform} stability asked for in Theorem \ref{ddapp}. So we
consider the following setup. Let $b\in\R[X,Y]$ be a polynomial of
degree $d>0$. We assume that the highest degree homogeneous part
of $b$ factors as
$$\prod_{i=1}^d \left(r_iX + s_iY\right),$$ where all the
$(r_i:s_i)$ are pairwise disjoint points of $\mathbb{P}^1(\R).$ In
particular, $b$ is square free. Let $C$ denote the affine curve in
$\A^2$ defined by $b$ and $\widetilde{C}$ its projective closure
in $\mathbb{P}^2$. So $\widetilde{C}$ is defined by $\tilde{b}$,
the homogenization of $b$ with respect to the new variable $Z$.
The assumption on the highest degree part of $b$ implies that all
the points at infinity of $b$, namely
$$P_1=(-s_1:r_1:0),\ldots, P_d=(-s_d:r_d:0)\in \mathbb{P}^2,$$ are
real regular points (of the projective curve $\widetilde{C}$). So
the local rings of $\widetilde{C}$ at all these points are
discrete valuation rings  (a well known fact, see for example
\cite{fu},  Chapter 3). Indeed, the projective curve
$\widetilde{C}$ is the so called ''good completion" (see for
example \cite{p2}) of the affine curve $C$. We denote the
valuation corresponding to the local ring at $P_i$ by $\ord_i$.
For a polynomial $h\in\R[X,Y],$ we write $\ord_{P_i}(h)$ and mean
the value with respect to the valuation $\ord_{P_i}$ of
$h(\frac{X}{Z},\frac{Y}{Z})$ as a rational function on
$\tilde{C}$.

We start with the following result:

\begin{proposition}\label{curve}
Let $b, C$ and $\widetilde{C}$ be as above. Suppose
$$\ord_{P_i}(h)\geq -l$$ for some $h\in\R[X,Y], l\in\N$ and all
$i$. Then there is some $h'\in\R[X,Y]$ with $\deg(h')\leq l$ and
$h\equiv h'\mod (b).$
\end{proposition}
\begin{proof} Let $m$ be the degree of $h$ and
$\tilde{b}=Z^d b(\frac{X}{Z},\frac{Y}{Z})$ as well as
$\tilde{h}=Z^m h(\frac{X}{Z},\frac{Y}{Z})$ be the homogenization
of $b$ and $h$, respectively. Assume without loss of generality
$$P_1=(1:y:0)$$ for some $y\in \R.$

For any homogeneous polynomial $g$ in the variables $X,Y,Z$ we
have
$$0\leq \ord_{P_1}\left(\frac{g}{X^{\deg(g)}}\right)=I(P_1;\tilde{b}\cap g),$$
where $I$ denotes the intersection number. This is \cite{fu},
Chapter 3.3.

As
$$\ord_{P_1}(h)=\ord_{P_1}\left(\frac{\tilde{h}}{X^m}\right)-m\cdot\ord_{P_1}\left(\frac{Z}{X}\right),$$
we have

\begin{align*} -l&\leq \ord_{P_1}(h)\\ &=
I(P_1;\tilde{b}\cap\tilde{h}) -m \cdot I(P_1;\tilde{b}\cap Z)\\
&\leq I(P_1;\tilde{b}\cap\tilde{h}) -m .
 \end{align*}

 Now whenever $m\geq l+1$, then $$1\leq
I(P_1;\tilde{b}\cap\tilde{h}),$$  so  $\tilde{h}$ must vanish at
$P_1$.

 The same argument applies to all points at infinity of
$b$. So if $m\geq l+1$, then the highest degree part of $b$
divides the highest degree part of $h$ in $\R[X,Y]$. Thus $h$ can
be reduced modulo $b$ to a polynomial $h'$ of strictly smaller
degree.
\end{proof}
In the following proposition, the pure closedness and stability
result follows from \cite{s1}, Proposition 6.5.

\begin{proposition}\label{simstab} Let $M=\QM(f_1,\ldots,f_s)\subseteq \R[X,Y]$ be a finitely
generated quadratic module. Let $b\in\R[X,Y]$ be a polynomial
whose highest degree part factors as above. For some $r\in\R$
assume that all the points at infinity of the curve $C_{r}$
defined by $b=r$ lie in the closure of $\Se(M)\cap C_{r}(\R)$.
Then the finitely generated quadratic module
$$M+ (b-r)=\QM(f_1,\ldots,f_s,b-r,r-b)$$ is closed and stable, with a stability map that
depends only on $b$ and $f_1,\ldots,f_s$, but not on $r$.
\end{proposition}
\begin{proof}  Without loss of generality, let
$P_1=(1:y:0)$ be a point   at infinity of $C_{r}$. Denote by
$\ord_{P_1}$ the valuation with respect to the local ring of
$\widetilde{C_{r}}$ at $P_1$. Let $h\in\R[X,Y]$ have degree $m$,
and let $\tilde{h}$ as well as $ \widetilde{b-r}$ be the
homogenizations, as in the previous proof. Then
\begin{align*}\ord_{P_1}(h) &=
\ord_{P_1}\left(\frac{\tilde{h}}{X^m}\right)-m\cdot\ord_{P_1}\left(\frac{Z}{X}\right)\\
&\geq -m\cdot I\left(P_1; \widetilde{b-r}\cap Z\right)\\&= -m\cdot
I\left(P_1;\tilde{b}\cap Z\right),\end{align*} where the last
equality uses property (7) in \cite{fu}, p. 75, for intersection
numbers. So there is some $N$, not depending on $r$, such that
$$\ord_{P}(h)\geq -m\cdot N$$ for all the  points of infinity of $C_{r}$.

 Now the proof of Proposition 6.5 from
\cite{s2} shows that whenever $h\in \overline{M+(b-r)}$, then we
can find a representation
\begin{align}\label{repr}h=\sum_{i=0}^s \si_if_i +
g\cdot(b-r)\end{align} with sums of squares $\si_i$ built of
polynomials that have order greater than $-m\cdot N$ in all points
at infinity of $C_{r}$. Applying Proposition \ref{curve} we can
reduce these elements modulo $b-r$ and obtain a representation as
in (\ref{repr}) with sums of squares of elements of degree less or
equal to $m\cdot N$. So of course also the degree of $g$ is
bounded suitably,  independent of $r$. This shows that the
stability map does not depend on $r$.
\end{proof}

So the following Theorem is an immediate consequence of Theorem
\ref{ddapp} and Proposition \ref{simstab}.

\begin{theorem}\label{infcurve} Let $M\subseteq\R[X,Y]$ be a finitely
generated quadratic module. Let $b\in\R[X,Y]$ with $\La-b,b-\la\in
M$ for some $\la\leq \La$, and assume the highest degree part of
$b$ factors as above.  Suppose that for all $r\in[\la,\La]$ all
the
 points at infinity of the curve $C_{r}$ defined by
$b=r$ lie in the closure of $\Se(M)\cap C_{r}(\R)$. Then
$$M^{\dd}=\overline{M}$$ holds. If all the fibre modules $M+(b-r)$
have (SMP) in addition, then $M$ has the $\dd$-property.
\end{theorem}

We give some explicit examples for these last results.

\begin{example}\label{einsdurch}
We look at the semi-algebraic set in $\R^2$ defined by the
inequalities $$ 0\leq x, 0\leq y \mbox{ and } xy\leq 1:$$
\begin{center}\bf\includegraphics[scale=0.3]{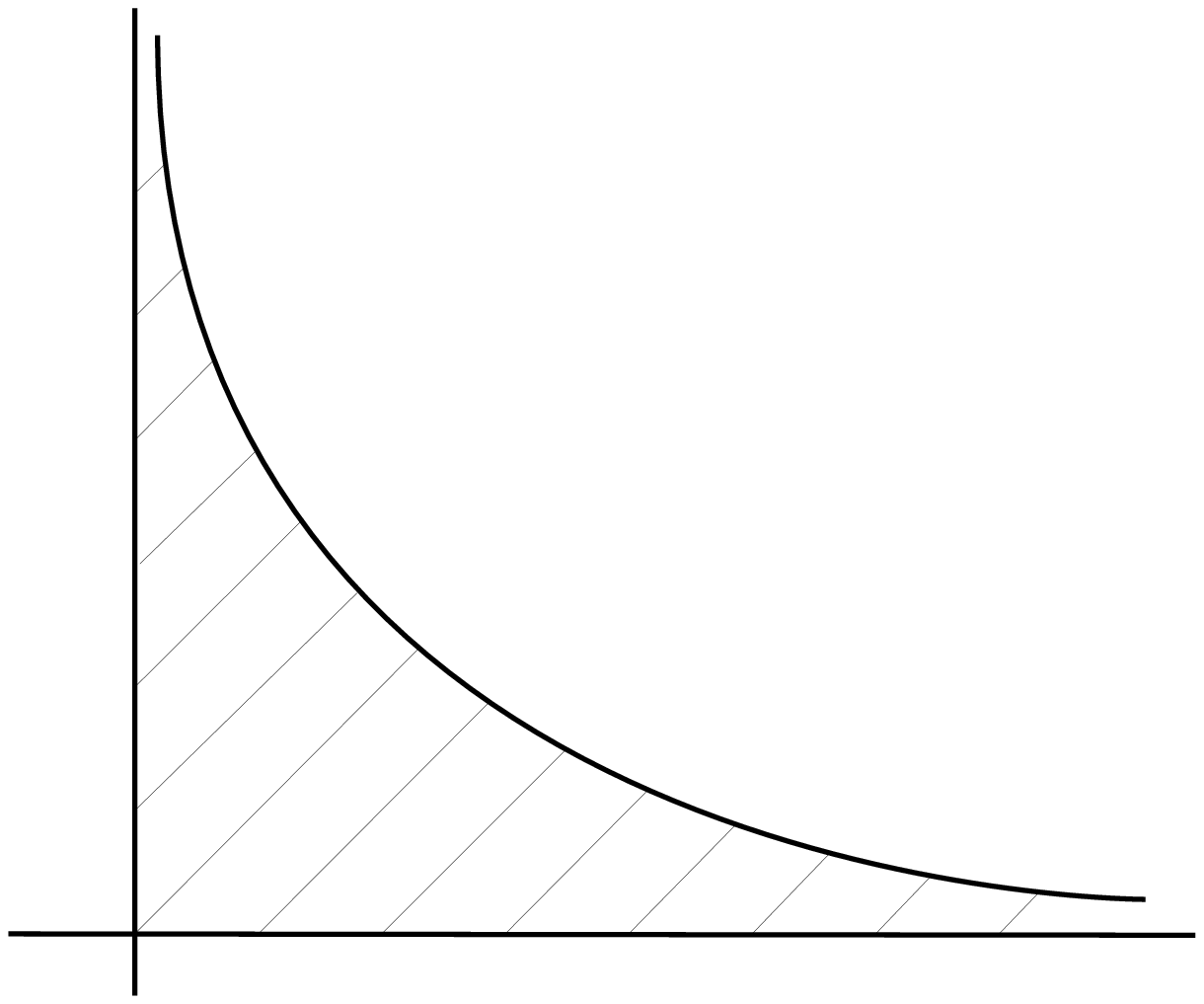}\end{center}
A lot of interesting phenomena can be observed for this set. There
are different quadratic modules describing it, we consider the
following ones:

\begin{align*} M_1&:=\QM\left(X,Y,1-XY\right) \\ M_2&:=\QM\left(X,Y,XY,1-XY\right)
\\M_3&:=\QM\left(X,Y^3,XY,1-XY\right)\\ P&:=\PO\left(X,Y, 1-XY\right)
 \end{align*}

The quadratic module $M_1$ is stable; one checks that no
cancellation of highest degree terms can occur in a sum $$\si_0 +
\si_1X +\si_2Y +\si_3(1-XY).$$ So by \cite{s}, Theorem 5.4, $M_1$
does not have (SMP).

To the quadratic module $M_2$ we can apply Theorem \ref{infcurve}
with the  polynomial $b=XY$:   we have $b,1-b\in M_2$. For
$r\in[0,1]$, the finitely generated quadratic module
$$\QM(X,Y,XY,1-XY) + (XY-r)=\QM(X,Y) + (XY-r)$$  is  saturated.
 This is an
easy calculation for $r>0$; for $r=0$ it is Example 3.26 from
\cite{p2}. So $M_2$ has the $\dd$-property, and in particular
(SMP).

Note that the fibre modules of $M_1$ and $M_2$ are the same:
$$M_1+(XY-r)= M_2+(XY-r)$$ for all $r\in[0,1].$ As $M_1$
does not have the $\dd$-property, this shows that the condition
$\La-b,b-\la\in M$ in Theorem \ref{infcurve}, as well as the
corresponding conditions in Theorems \ref{ddapp} and \ref{main}
can not be omitted.

Now consider $M_3$. The quadratic module $\QM(Y^3)\subseteq\R[Y],$
obtained by factoring out the ideal $(X)$, does not have (SMP)
(see for example \cite{km}). So in view of Proposition 4.8 from
\cite{s},
 $M_3$ does also not have (SMP). On the
other hand, we can still apply Theorem \ref{infcurve} with $b=XY$,
and obtain
$$M_3^{\dd}=\overline{M_3}.$$

Last, the preordering $P$ obviously contains $M_2$ and therefore
also has the $\dd$-property. This solves the question posed in
\cite{kms}, Example 8.4.
\end{example}

\begin{example} We consider the semi-algebraic set defined by the
inequalities $$0\leq x(x+y)(x-y) -xy \leq 1:$$
\begin{center}\bf\includegraphics[scale=0.2]{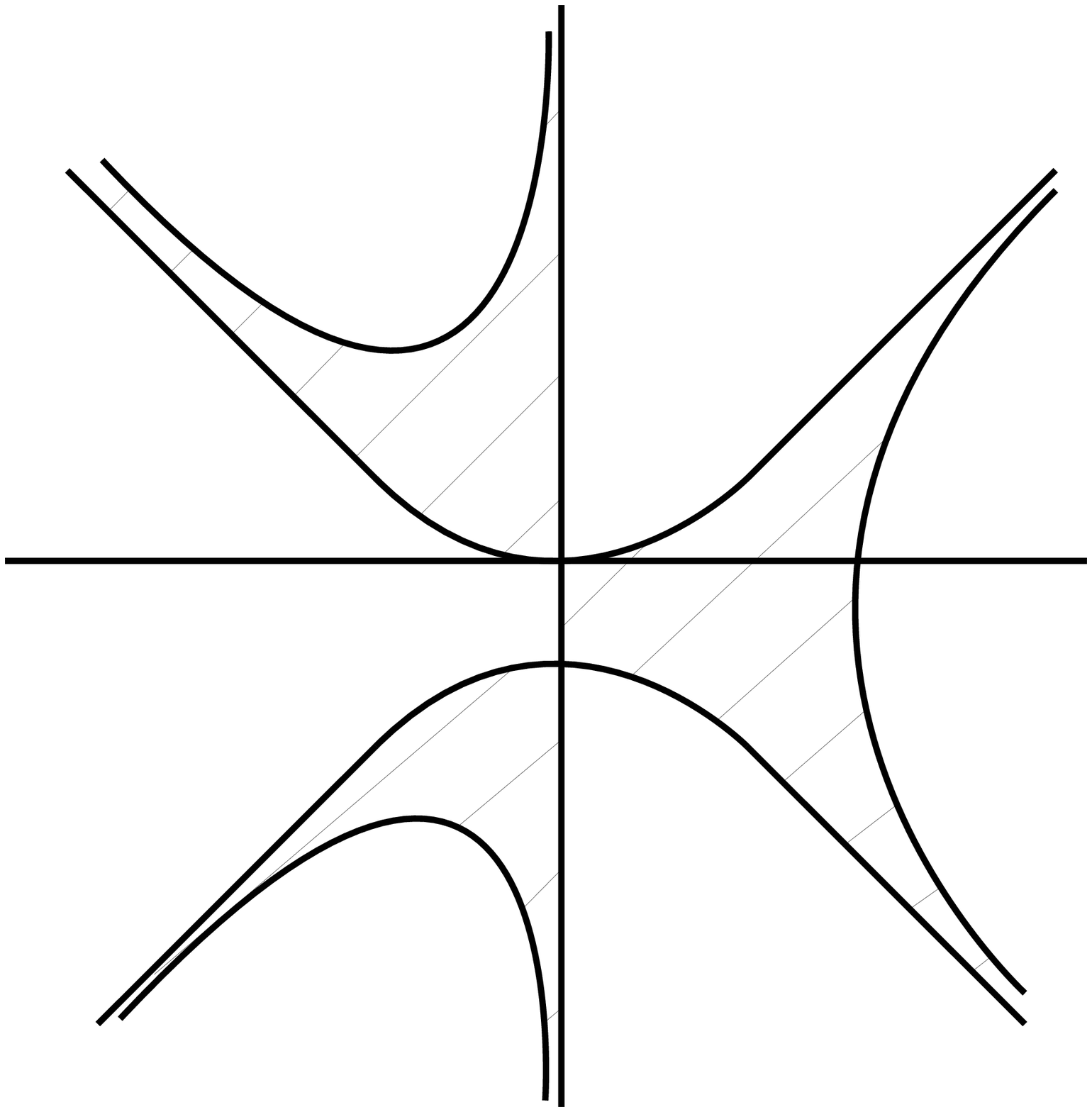}\end{center}
We can apply Theorem \ref{infcurve} to the quadratic module
$$M=\QM(b,1-b),$$ where $b=X(X+Y)(X-Y)-XY$. We use $b$ as the
bounded polynomial and obtain $$M^{\dd}=\overline{M}.$$ However,
$M$ does not have (SMP). Indeed, the quadratic module
$$M+ (b)$$ does not have (SMP). This follows from \cite{p2},
Theorem 3.17 together with \cite{s1}, Proposition 6.5. So in view
of Proposition 4.8 from \cite{s}, $M$ does not have (SMP).
\end{example}

\section{Application to Polynomial Optimization}
We want to explain how the result of Theorem  $\ref{ddapp}$,
together with the observation from Remark \ref{rem} (3), can be
used for constrained polynomial optimization. We take a similar
approach to the one in \cite{l}, see also \cite{m} Chapter 10 and
\cite{sw2} for a nice account of the topic. However, our approach
does not only  apply to compact semi-algebraic sets.

Assume $A=\R[\underline{X}]$, so $\V_A=\R^n$. Consider a finitely
generated quadratic module
$$M=\QM(f_1,\ldots,f_s)\subseteq\R[\underline{X}]$$ with associated semi-algebraic set
$\Se=\Se(M)$, which has the property
$$f\geq 0 \mbox{ on } \Se \Rightarrow \exists q \mbox{ s.t. } f+\ep q\in M \mbox{ for all }
\ep>0,$$ where $q$ can be chosen to only depend on the degree of
$f$ (in other words: for any nonnegative polynomial of the same
degree, the same $q$ works). Without loss of generality $q$ can be
chosen to be from $M$ (see \cite{km}, the note following
Proposition 1.3). Note that for example the quadratic modules from
Corollary \ref{cylinder} and Theorem \ref{infcurve} have this
stronger property. Note also that if a compact semi-algebraic set
is described by a preordering, or more general, by an archimedean
quadratic module $M$, the above condition holds, as every strictly
positive polynomial belongs to $M$. So $q=1$ works for every
nonnegative polynomial.

Given some $f\in\R[\underline{X}]$, one wants to calculate the
infimum of $f$ on $\Se$, denoted by $f_*$. This is usually a hard
problem, but for compact semi-algebraic sets $S$, Lasserre
\cite{l} provided a sequence of semi-definite programs (which are
much easier to solve), whose optimal values converge to $f_*$.

Now take the polynomial $q$ from above corresponding to the degree
of $f$. It can be obtained explicitly in the case of Theorem
\ref{main} from the proof. For example, in Corollary
\ref{cylinder}, $q$ can be chosen to be the sum of the elements
$$\left(\frac{Y^j-1}{2}\right)^2 \mbox{ and }
\left(\frac{Y^j+1}{2}\right)^2,$$ where $j$ runs from $0$ to the
$Y$-degree of $f$.

For fixed $\ep>0$ and $d\in\N$, we consider
$$F_{\ep,d}:=\sup\left\{r\mid f-r +\ep q \in M_{d}\right\}$$ and
$$F_{\ep}:=\sup\left\{r\mid f-r +\ep q \in M\right\}.$$ Here,
$M_d$ denotes the set of all elements from $M$ that can be
represented by $f_1,\ldots,f_s$ and sums of squares of degree
$\leq d$.

 We
obviously have $F_{\ep,d}\leq F_{\ep}$ for all $d$. Furthermore,
each $F_{\ep,d}$ can be obtained by solving a semi-definite
program, which can be done efficiently; see \cite{l,m,sw2}.

\begin{proposition}
The sequence $(F_{\ep,d})_{d\in\N}$ converges monotonically
increasing to $F_{\ep}$.
\end{proposition}
\begin{proof}
It is clear that the sequence is monotonically increasing. Now
take some $r$ which is feasible for $F_{\ep}$, which means that
$f-r+\ep p$ belongs to $M$ (if no such $r$ exists, then the
statement is also clear, as all values are $-\infty$). But then
$f-r+\ep p$ belongs to $M_{d}$ for some big enough $d$. So
$F_{\ep,d}\geq r$, which proves the statement.
\end{proof}

Now suppose $f\geq r$ on $\Se$ for some $r\in\R$. Then $f-r$ is
nonnegative on $\Se$ and so $$f-r+\ep q\in M$$ holds for all
$\ep>0$.  This shows $F_{\ep}\geq f_*$ for all $\ep>0$. We have
used here that subtracting $r$ from $f$ does not change the
degree, and therefore we can use the polynomial $q$, no matter how
big or small $r$ is. This could fail if $M$ only has the
$\dd$-property, not the stronger version we assume.

\begin{proposition}
For $\ep \searrow 0$, the sequence $(F_{\ep})_{\ep}$ converges
monotonically decreasing to $f_*$.
\end{proposition}
\begin{proof} From the fact that $q$ is in $M$ it is clear
that the sequence is monotonically decreasing. Now suppose $f_*$
is finite and  $F_{\ep}\geq f_* +\de$ for some $\de>0$ and all
$\ep>0$. This means that $f_*+\frac{\de}{2}$ is feasible for all
$F_{\ep}$ and so $$f -f_*-\frac{\de}{2}+\ep q \in M $$ holds for
all $\ep>0$. This clearly implies $f-f_*-\frac{\de}{2}\geq 0$ on
$\Se$ and so $f\geq f_* +\frac{\de}{2}$ on $\Se$, a contradiction.

If $f_*=-\infty$, the same argument applies by assuming
$F_{\ep}\geq N$ for some $N\in\R$ and all $\ep$.

\end{proof}

So combining these results, we get the following:

\begin{theorem}
For every $f\in\R[\underline{X}]$ there is a sequence
$(m_i)_{i\in\N}$ of natural numbers, such that the sequence
$\left(F_{\frac{1}{i},m_{i}}\right)_i$ converges to $f_*$.
\end{theorem}

\end{document}